\documentclass{amsart}
\usepackage{amsmath,amssymb}
\usepackage[dvips]{graphics}
\usepackage[all]{xy}
\newtheorem{Theorem}{Theorem}[section]
\newtheorem{Lemma}[Theorem]{Lemma}

\newtheorem{Corollary}[Theorem]{Corollary}

\theoremstyle{definition}

\theoremstyle{remark}
\newtheorem{Remark}[Theorem]{Remark}

\def\labelenumi{\rm ({\makebox[0.8em][c]{\theenumi}})}
\def\theenumi{\arabic{enumi}}
\def\labelenumii{\rm {\makebox[0.8em][c]{(\theenumii)}}}
\def\theenumii{\roman{enumii}}

\renewcommand{\thefigure}
{\arabic{section}.\arabic{figure}}
\makeatletter
\@addtoreset{figure}{section}
\def\@thmcountersep{-}
\makeatother

\numberwithin{equation}{section}

\begin{document} 

\title[A homotopy classification of two-component spatial graphs]{A homotopy classification of two-component spatial graphs up to neighborhood equivalence}

\author{Atsuhiko Mizusawa}
\address{Department of Mathematics, Faculty of Fundamental Science and Engineering, Waseda University, 3-4-1 Okubo, Shinjuku-ku, Tokyo 169-8555, Japan}
\email{a\symbol{"5F}mizusawa@aoni.waseda.jp}

\author{Ryo Nikkuni}
\address{Department of Mathematics, School of Arts and Sciences, Tokyo Woman's Christian University, 2-6-1 Zempukuji, Suginami-ku, Tokyo 167-8585, Japan}
\email{nick@lab.twcu.ac.jp}
\thanks{The second author was partially supported by Grant-in-Aid for Scientific Research (C) (No. 24540094), Japan Society for the Promotion of Science.}

\subjclass{Primary 57M15; Secondary 57M25}

\date{}

\dedicatory{This article is dedicated to Professors Taizo Kanenobu, Yasutaka Nakanishi and Makoto Sakuma on their 60th birthdays.}

\keywords{Spatial graph, Linking number, Delta move, Handlebody-link}

\begin{abstract}
A neighborhood homotopy is an equivalence relation on spatial graphs which is generated by crossing changes on the same component and neighborhood equivalence. We give a complete classification of all $2$-component spatial graphs up to neighborhood homotopy by the elementary divisor of a linking matrix with respect to the first homology group of each of the connected components. This also leads a kind of homotopy classification of $2$-component handlebody-links. 
\end{abstract}

\maketitle

\section{Introduction} 

Throughout this paper we work in the piecewise linear category. An embedding of a graph into the $3$-sphere ${\mathbb S}^{3}$ is called a {\it spatial embedding} of the graph and the image is called a {\it spatial graph}. We say that two spatial graphs $G$ and $G'$ are {\it ambient isotopic} if there exists an orientation-preserving self-homeomorphism $\Phi$ on ${\mathbb S}^{3}$ such that $\Phi(G)=G'$. A graph is said to be {\it planar} if there exists an embedding of the graph into the $2$-sphere, and a spatial embedding of a planar graph is said to be {\it trivial} if it is ambient isotopic to an embedding of the graph into a $2$-sphere in ${\mathbb S}^{3}$. Such an embedding is unique up to ambient isotopy \cite{M69}. On the other hand, let us denote the regular neighborhood of a spatial graph $G$ in ${\mathbb S}^{3}$ by $N(G)$. Then, two spatial graphs $G$ and $G'$ are said to be {\it neighborhood equivalent} if there exists an orientation-preserving self-homeomorphism $\Phi$ on ${\mathbb S}^{3}$ such that $\Phi(N(G))=N(G')$ \cite{S70}. Note that ambient isotopic two spatial graphs are homeomorphic as abstract graphs, but neighborhood equivalent two spatial graphs are not always homeomorphic. For a spatial graph $G$ and an edge $e$ of $G$ which is not a loop, we call the spatial graph obtained from $G-{\rm int}e$ by identifying the end vertices of $e$ the {\it edge contraction} of $G$ along $e$. A {\it vertex splitting} is the reverse of an edge contraction. Then it is known that two spatial graphs are neighborhood equivalent if they are transformed into each other by edge contractions, vertex splittings and ambient isotopies \cite{I08}. 

Two oriented links are said to be {\it link homotopic} if they are transformed into each other by crossing changes on the same component and ambient isotopies. It is well known that two oriented $2$-component links are link homotopic if and only if they have the same linking number \cite{M54}. Our purpose in this article is to generalize this fact to $2$-component spatial graphs from a viewpoint of neighborhood equivalence. We introduce the notion of {\it neighborhood homotopy} on spatial graphs as an equivalence relation which is generated by crossing changes on the same component and neighborhood equivalence; that is, two spatial graphs $G$ and $G'$ are neighborhood homotopic if they are transformed into each other by crossing changes between edges which belong to the same component, edge contractions, vertex splittings and ambient isotopies. Note that in the case of oriented links, neighborhood homotopy coincides with link homotopy. Moreover, we also introduce another equivalence relation on spatial graphs as follows. A {\it Delta move} is a local move on a spatial graph as illustrated in Fig. \ref{deltamove} \cite{Matveev87}, \cite{MN89}. We say that two spatial graphs are {\it Delta neighborhood equivalent} if they are transformed into each other by Delta moves, edge contractions, vertex splittings and ambient isotopies. 

\begin{figure}[htbp]
      \begin{center}
\scalebox{0.375}{\includegraphics*{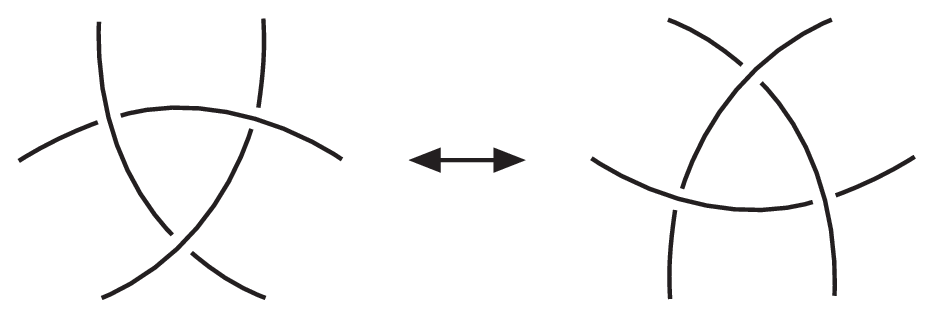}}
      \end{center}
   \caption{}
  \label{deltamove}
\end{figure} 

In \cite{Mi12}, the first author introduced a sequence of invariant nonnegative integers for $2$-component spatial graphs under neighborhood equivalence as follows. Let $G=G_{1}\cup G_{2}$ be a $2$-component spatial graph. Let ${\mathcal Z}=\left\{z_{1},z_{2},\ldots,z_{m}\right\}$ be a basis of $H_{1}(G_{1};{\mathbb Z})$ and ${\mathcal W}=\left\{w_{1},w_{2},\ldots,w_{n}\right\}$ a basis of $H_{1}(G_{2};{\mathbb Z})$. Let $M_{G}\left({\mathcal Z},{\mathcal W}\right)$ be the $(m,n)$-matrix whose $(i,j)$-entry is the {\it linking number} ${\rm lk}(z_{i},w_{j})$ in ${\mathbb S}^{3}$. Then the sequence of elementary divisors $d_{1},d_{2},\ldots ,d_{l}$ $(d_{i}\in {\mathbb Z}_{> 0},\ d_{i}|d_{i+1}\ (i=1,2,\ldots, l-1))$ of $M_{G}\left({\mathcal Z},{\mathcal W}\right)$ is an invariant under neighborhood equivalence. We define ${\rm Lk}(G_{1},G_{2})$ by the sequence $\left\{d_{1},d_{2},\ldots,d_{l}\right\}$ if $l\ge 1$ and otherwise $0$. Now we state our main theorem. 

\begin{Theorem}\label{main}
Let $G=G_{1}\cup G_{2}$ and $G'=G'_{1}\cup G'_{2}$ be two $2$-component spatial graphs satisfying with $H_{1}(G_{i};{\mathbb Z})\cong H_{1}(G'_{i};{\mathbb Z})\ (i=1,2)$. Then the following are equivalent. 
\begin{enumerate}
\item $G$ and $G'$ are neighborhood homotopic. 
\item $G$ and $G'$ are Delta neighborhood equivalent. 
\item ${\rm Lk}(G_{1},G_{2})={\rm Lk}(G'_{1},G'_{2})$. 
\end{enumerate}
\end{Theorem}

A {\it handlebody-link} in ${\mathbb S}^{3}$ is the image of an embedding of mutually disjoint handlebodies into ${\mathbb S}^{3}$. Two handlebody-links $L$ and $L'$ are said to be {\it equivalent} if there exists an orientation-preserving self-homeomorphism $\Phi$ on ${\mathbb S}^{3}$ such that $\Phi(L)=L'$. Note that two handlebody-links are equivalent if and only if their spines are neighborhood equivalent as spatial graphs. We say that two handlebody-links are {\it homotopic} if their spines are neighborhood homotopic as spatial graphs. For a $2$-component handlebody-link $L=V_{1}\cup V_{2}$, we can define ${\rm Lk}(V_{1},V_{2})$ in the same way as ${\rm Lk}(G_{1},G_{2})$ for a $2$-component spatial graph $G=G_{1}\cup G_{2}$ \cite{Mi12}. Then by Theorem \ref{main}, we immediately have the following. 

\begin{Corollary}\label{main_cor}
Let $L=V_{1}\cup V_{2}$ and $L'=V'_{1}\cup V'_{2}$ be two $2$-component handlebody-links satisfying with $H_{1}(V_{i};{\mathbb Z})\cong H_{1}(V'_{i};{\mathbb Z})\ (i=1,2)$. Then $L$ and $L'$ are homotopic if and only if ${\rm Lk}(V_{1},V_{2})={\rm Lk}(V'_{1},V'_{2})$. 
\end{Corollary}

\begin{Remark}
In \cite{F08}, Fleming introduced ``Milnor invariants'' for spatial graphs which is invariant under crossing changes on the same component and ambient isotopies. Since his invariants are derived from the fundamental group of spatial graph exterior, they also are invariant under neighborhood homotopy. See also \cite{M54} for Milnor's link homotopy invariants. 
\end{Remark}

In the next section, we show some lemmas which are needed later. We prove Theorem \ref{main} in section $3$. 

\section{Local moves on spatial graphs} 

Let $B^{3}$ be the oriented unit $3$-ball. Let $T$ be a tangle in $B^{3}$ and $A$ a disjoint union of arcs in $\partial B$ with $\partial T = \partial A$ as illustrated in Fig. \ref{Hopf_chord}. Let $G$ be a spatial graph. Let $\psi_{i}:B^{3}\to {\mathbb S}^{3}$ be an orientation-preserving embedding for $i=1,2,\ldots,k$. Note that each $\psi_{i}(T\cup A)$ is a {\it Hopf link} in ${\mathbb S}^{3}$. Let $b_{i,p}$ be a $2$-disk which is embedded in ${\mathbb S}^{3}$ for $i=1,2,\ldots,k$ and $p=1,2$. Suppose that $\psi_{i}(B^{3})\cap G=\emptyset$ for each $i$, $\psi_{i}(B^{3})\cap \psi_{j}(B^{3})=\emptyset$ for $i\neq j$ and $b_{i,p}\cap b_{j,q}=\emptyset$ for $(i,p)\neq (j,q)$. Suppose that $b_{i,p}\cap G=\partial b_{i,p}\cap G$ is an arc away from the vertices of $G$ for each $i$ and $p$. Suppose that $b_{i,p}\cap \psi_{j}(B^{3})=\emptyset$ for $i\neq j$ and $b_{i,p}\cap \psi_{i}(B^{3})=\partial b_{i,p}\cap \psi_{i}(B^{3})$ is a component of $\psi_{i}(A)$ for each $i$ and $p$. Let $H$ be a spatial graph satisfying with the followings: 
\begin{eqnarray}
&& G\setminus \bigcup_{i,p}b_{i,p} = H\setminus \bigcup_{i,p}b_{i,p}\cup\bigcup_{i}\psi_{i}(T),\label{g1}\\
&& H= G \cup {\bigcup_{i,p}\partial b_{i,p}}\cup {\bigcup_{i}\psi_{i}(T)}\setminus {\bigcup_{i,p}{\rm int}(G\cap b_{i,p})}\cup {\bigcup_{i}\psi_{i}({\rm int}A)}. \label{g2}
\end{eqnarray}
Then $H$ is called a {\it band sum of Hopf links} and $G$. The union $b_{i,1}\cup b_{i,2}\cup \psi_{i}(B^{3})$ is called a {\it Hopf chord}, and $b_{i,1}$ and $b_{i,2}$ are called the {\it associated bands} of the Hopf chord. An edge $e$ of $G$ is called an {\it associated edge} of the Hopf chord if $e$ has intersection with the associated bands. 

\begin{figure}[htbp]
      \begin{center}
\scalebox{0.45}{\includegraphics*{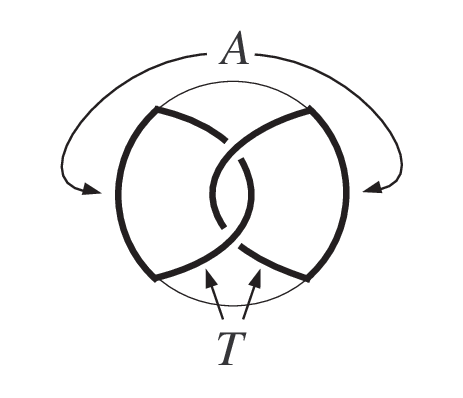}}
      \end{center}
   \caption{}
  \label{Hopf_chord}
\end{figure} 

Note that a crossing change is realized by a band sum of a single Hopf link, see Fig. \ref{Hopf_chord2}. Since any two spatial graphs which are homeomorphic as abstract graphs are transformed into each other by crossing changes and ambient isotopies, we have the following lemma. 

\begin{Lemma}\label{forklore} {\rm (\cite{S69}, \cite{Yama90}, \cite{TY02})} 
Let $G$ and $H$ be two spatial graphs which are homeomorphic as abstract graphs. Then $H$ is a band sum of Hopf links and $G$. 
\end{Lemma}

\begin{figure}[htbp]
      \begin{center}
\scalebox{0.375}{\includegraphics*{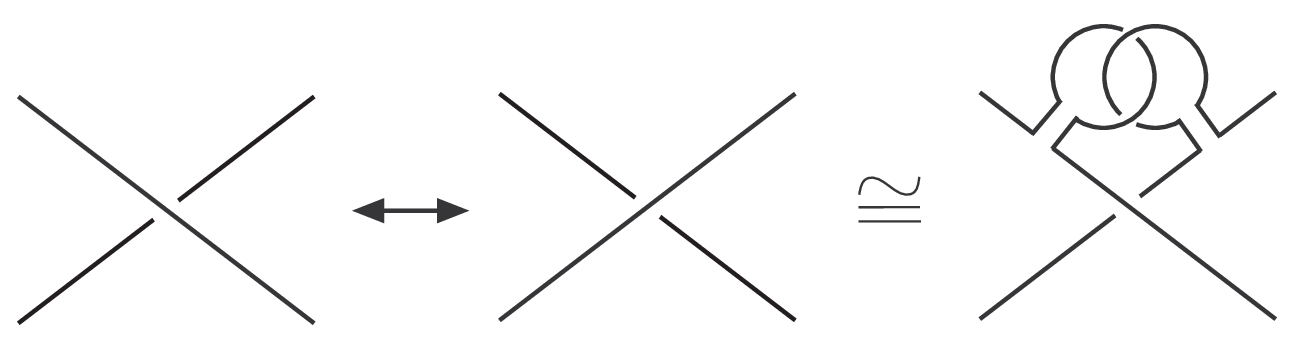}}
      \end{center}
   \caption{}
  \label{Hopf_chord2}
\end{figure} 

Then for a Delta move and a band sum of Hopf links, it is known the following. 

\begin{Lemma}\label{delta_lemma}
Each of the local moves (1), (2) and (3) illustrated in Fig. \ref{Hopf_chord3} is realized by Delta moves and ambient isotopies. 
\end{Lemma}

\begin{proof}
See \cite[Lemma 2.2]{TY02}. 
\end{proof}

\begin{figure}[htbp]
      \begin{center}
\scalebox{0.35}{\includegraphics*{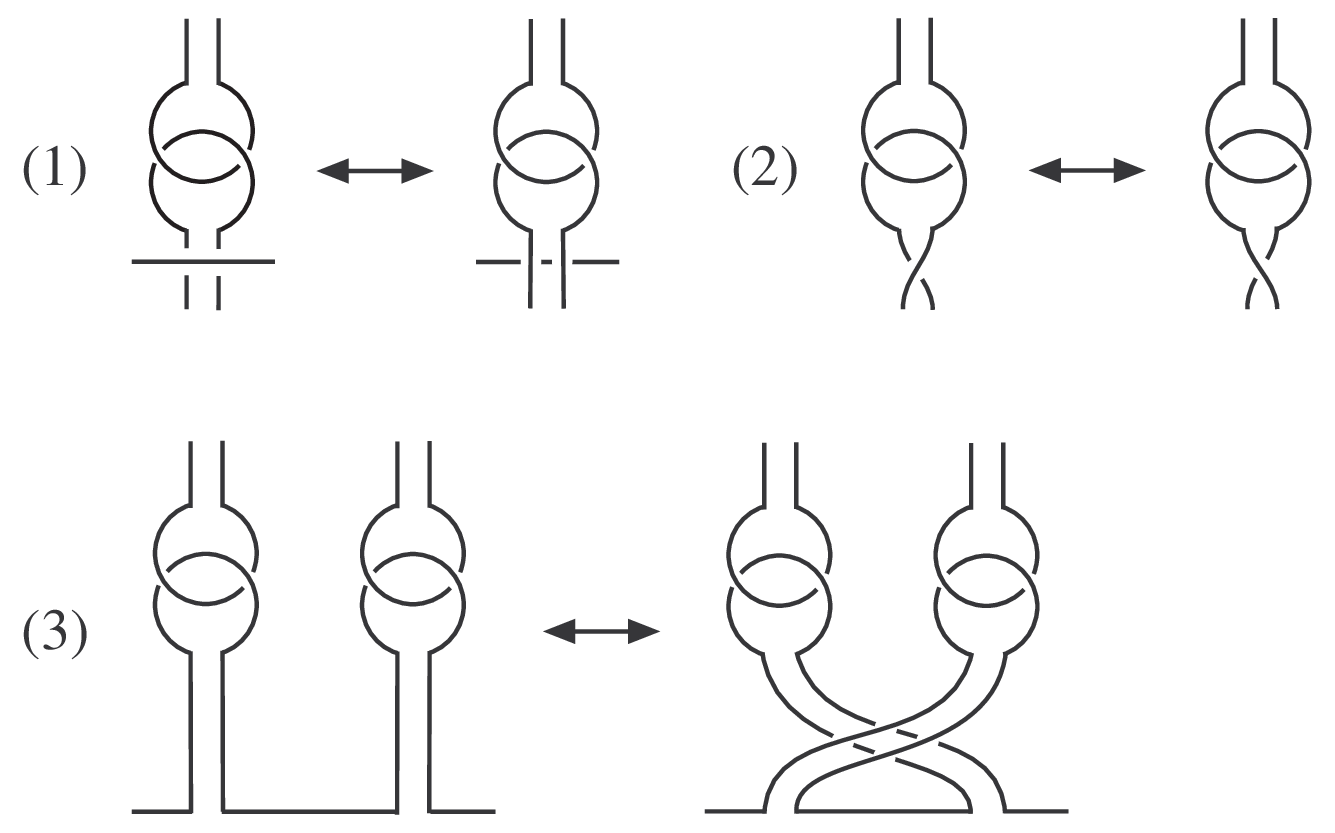}}
      \end{center}
   \caption{}
  \label{Hopf_chord3}
\end{figure} 

Moreover, we also have the following. 

\begin{Lemma}\label{delta_lemma2}
Each of the local moves (1), (2), $\ldots$, (6) illustrated in Fig. \ref{Hopf_cancel} is realized by an ambient isotopy. 
\end{Lemma}

\begin{proof}
(1), (2), $\ldots$, (5) are clear. In the case of (6), see Fig. \ref{Hopf_chord4}. 
\end{proof}

\begin{figure}[htbp]
      \begin{center}
\scalebox{0.35}{\includegraphics*{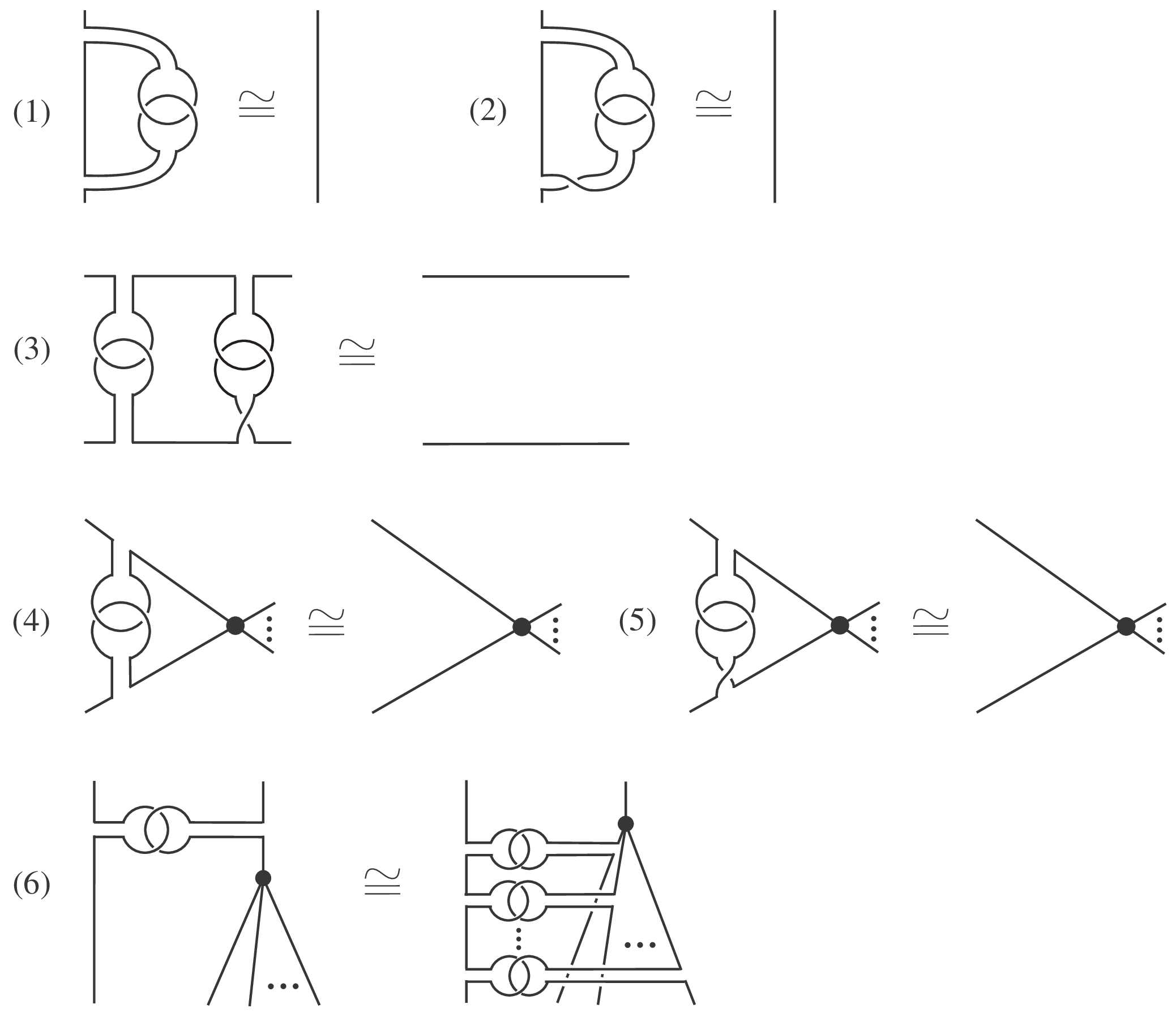}}
      \end{center}
   \caption{}
  \label{Hopf_cancel}
\end{figure} 
\begin{figure}[htbp]
      \begin{center}
\scalebox{0.3}{\includegraphics*{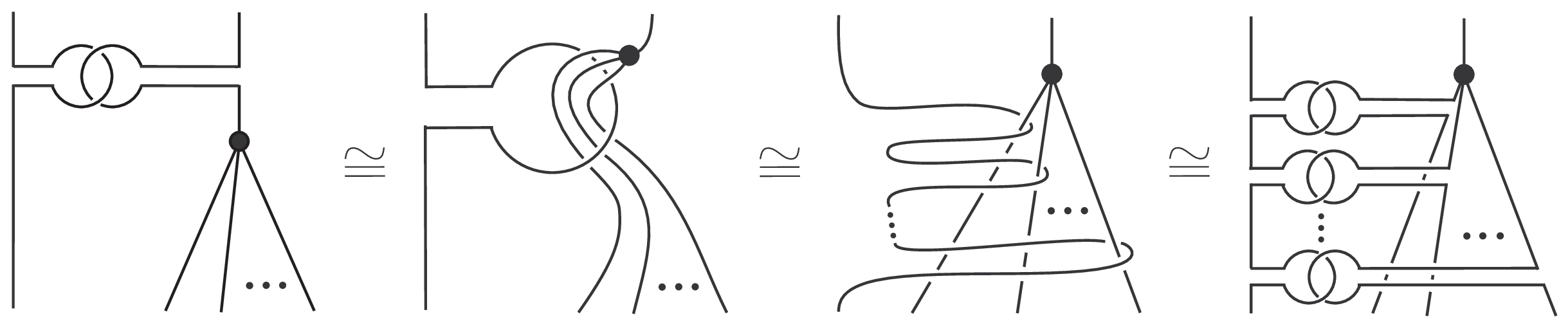}}
      \end{center}
   \caption{}
  \label{Hopf_chord4}
\end{figure} 

Now we show that neighborhood homotopy implies Delta neighborhood equivalence. Namely we have the following. 

\begin{Lemma}\label{delta_nh1}
If two spatial graphs are neighborhood homotopic then they are Delta neighborhood equivalent. 
\end{Lemma}

\begin{proof}
Let $G$ be a spatial graph. We show that a single crossing change on the same component of $G$ is realized by Delta moves, edge contractions, vertex splittings and ambient isotopies. Let $G'$ be a spatial graph which is obtained from $G$ by a single crossing change on the same component $G_{1}$ of $G$. Then by Lemma \ref{forklore}, $G'$ is a band sum of a Hopf link and $G$, where the associated edges of the Hopf chord belong to $G_{1}$. Let $T_{1}$ be a spanning tree of $G_{1}$. Then by sliding the roots of the associated bands of the Hopf chord along $T_{1}$ by using a move in Fig. \ref{Hopf_cancel} (6) if necessary, we can regard $G'$ as a band sum of Hopf links and $G$, where the associated edges of each of the Hopf chords do not belong to $T_{1}$. Let $G''$ be the spatial graph which is obtained from $G'$ by contracting all edges of $T_{1}$. Note that $G''$ has a spatial bouquet as a component, and the associated edges of each of the Hopf chords belong to the bouquet. Then by Lemma \ref{delta_lemma}, we deform $G''$ by Delta moves and ambient isotopies so that each of the Hopf chords is contained in a small $3$-ball as illustrated in Fig. \ref{Hopf_cancel0} (1), (2), (3) or (4). Then by the moves as illustrated in Fig. \ref{Hopf_cancel} (1), (2), (4) and (5), all of the Hopf chords can be removed. We finally obtain $G$ from $G''$ by restoring $T_{1}$ by suitable vertex splittings. 
\end{proof}

\begin{figure}[htbp]
      \begin{center}
\scalebox{0.35}{\includegraphics*{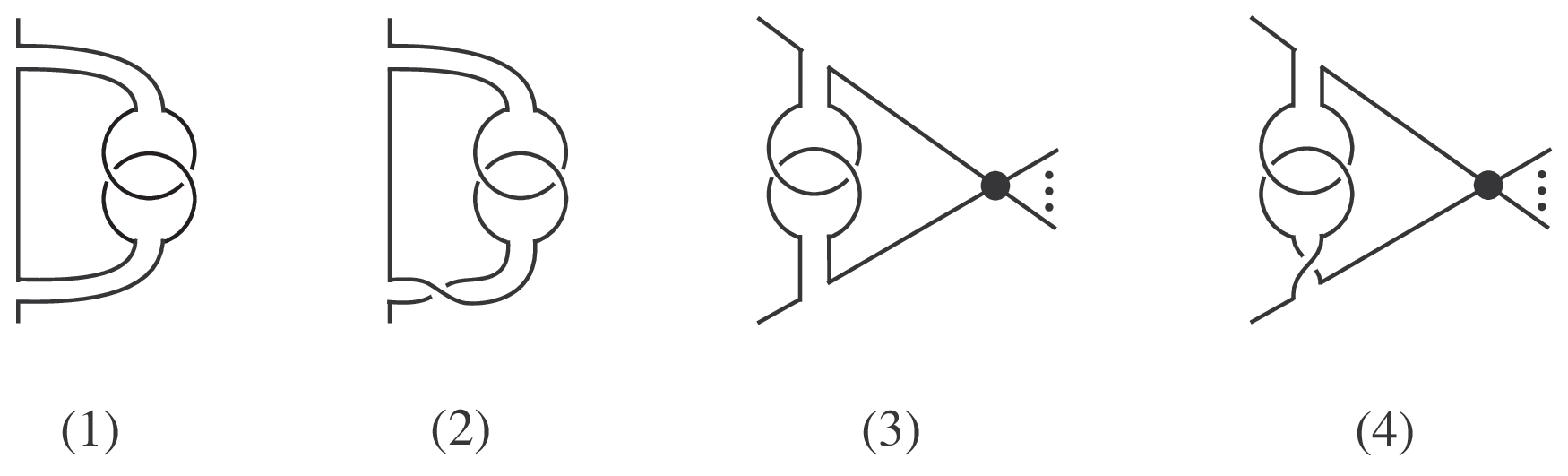}}
      \end{center}
   \caption{}
  \label{Hopf_cancel0}
\end{figure} 

In particular, the converse of Lemma \ref{delta_nh1} is also true in the case of $2$-component spatial graphs as follows. 

\begin{Lemma}\label{delta_nh2}
Two $2$-component spatial graphs are neighborhood homotopic if and only if they are Delta neighborhood equivalent. 
\end{Lemma}

\begin{proof}
By Lemma \ref{delta_nh1}, it is sufficient to show that if two $2$-component spatial graphs are Delta neighborhood equivalent then they are neighborhood homotopic. Let us consider a single Delta move on a $2$-component spatial graph. Then, there exist at least two of the three strings in the Delta move such that they belong to the same component. Then the Delta move is realized by two crossing changes on the same component and an ambient isotopy, see Fig. \ref{delta_hm}, where two strings which belong to the same component are expressed in bold lines.
\end{proof}

\begin{figure}[htbp]
      \begin{center}
\scalebox{0.325}{\includegraphics*{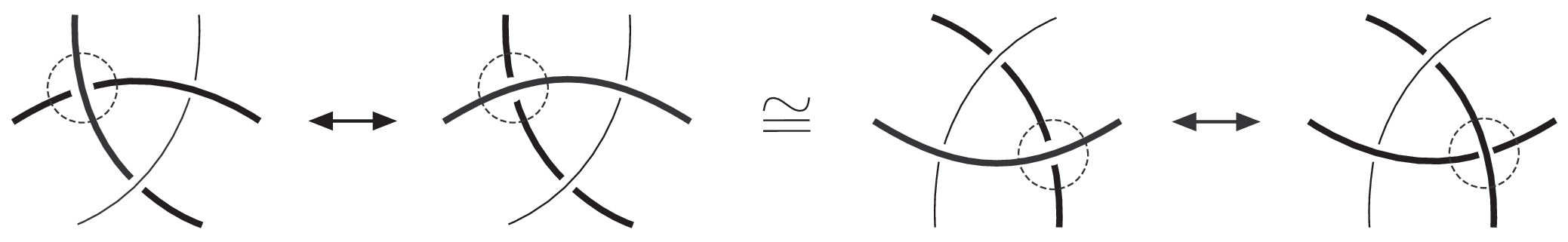}}
      \end{center}
   \caption{}
  \label{delta_hm}
\end{figure} 

\begin{Remark}\label{mu}
For a positive integer $n$, if $n\ge 3$ then there exist two spatial graphs which are Delta neighborhood equivalent but not neighborhood homotopic. Actually, Borromean rings can be undone by a single Delta move \cite{MN89} but not trivial up to link homotopy \cite{M54}. 
\end{Remark}

\section{Proof of Theorem \ref{main}} 

In this section we prove Theorem \ref{main}. 

\begin{proof}[Proof of Theorem \ref{main}]
By Lemma \ref{delta_nh2}, it follows that (1) and (2) are equivalent. In the following we show that (2) and (3) are equivalent. First we show that (2) implies (3). It is well known that a Delta move on a $2$-component oriented link does not change the linking number \cite{MN89} (the proof is the same as the proof of the fact that a Reidemeister move III does not change the linking number). Let ${\mathcal Z}=\left\{z_{1},z_{2},\ldots,z_{m}\right\}$ be a basis of $H_{1}(G_{1};{\mathbb Z})$ and ${\mathcal W}=\left\{w_{1},w_{2},\ldots,w_{n}\right\}$ a basis of $H_{1}(G_{2};{\mathbb Z})$. Note that both an edge contraction and a vertex splitting on $G$ do not change ${\rm lk}(z_{i},w_{j})$ through the isomorphism of the first homology. Moreover, since $z_{i}$ (resp. $w_{j}$) is represented by a homological sum of oriented knots in $G_{1}$ (resp. $G_{2}$), we see that ${\rm lk}(z_{i},w_{j})$ is represented by a sum of the linking numbers for some $2$-component constituent links in $G$. This implies that a Delta move on $G$ does not change ${\rm lk}(z_{i},w_{j})$. Thus it follows that there exist a basis ${\mathcal Z}'=\left\{z'_{1},z'_{2},\ldots,z'_{m}\right\}$ of $H_{1}(G'_{1};{\mathbb Z})$ and a basis ${\mathcal W}'=\left\{w'_{1},w'_{2},\ldots,w'_{n}\right\}$ of $H_{1}(G'_{2};{\mathbb Z})$ such that 
\begin{eqnarray*}
M_{G}\left({\mathcal Z},{\mathcal W}\right)=\left({\rm lk}(z_{i},w_{j})\right)=\left({\rm lk}(z'_{i},w'_{j})\right)=M_{G'}\left({\mathcal Z}',{\mathcal W}'\right).
\end{eqnarray*}
This implies that ${\rm Lk}(G_{1},G_{2})={\rm Lk}(G'_{1},G'_{2})$. 

Next we show that (3) implies (2). Assume that ${\rm Lk}(G_{1},G_{2})={\rm Lk}(G'_{1},G'_{2})$ is the sequence $\left\{d_{1},d_{2},\ldots ,d_{l}\right\}$ $(d_{i}\in {\mathbb Z}_{> 0},\ d_{i}|d_{i+1}\ (i=1,2,\ldots, l-1))$. In the following we deform $G=G_{1}\cup G_{2}$ into a certain canonical form by Delta moves, edge contractions, vertex splittings and ambient isotopies. Let $T_{i}$ be a spanning tree of $G_{i}$ $(i=1,2)$. Let $e_{1},e_{2},\ldots,e_{m}$ be all of the edges of $G_{1}$ which are not contained in $T_{1}$ and $f_{1},f_{2},\ldots,f_{n}$ all of the edges of $G_{2}$ which are not contained in $T_{2}$. Let ${\mathcal Z}=\left\{z_{1},z_{2},\ldots,z_{m}\right\}$ be the basis of $H_{1}(G_{1};{\mathbb Z})$ which are represented by $e_{1},e_{2},\ldots,e_{m}$, and ${\mathcal W}=\left\{w_{1},w_{2},\ldots,w_{n}\right\}$ the basis of $H_{1}(G_{2};{\mathbb Z})$ which are represented by $f_{1},f_{2},\ldots,f_{n}$. Let $B_{i}$ be a spatial bouquet obtained from $G_{i}$ by contracting all edges of $T_{i}$ ($i=1,2$). Note that all loops of $B_{1}$ (resp. $B_{2}$) can be regarded as $z_{1},z_{2},\ldots,z_{m}$ (resp. $w_{1},w_{2},\ldots,w_{n}$) through the isomorphism of the first homology. Let $U_{1}\cup U_{2}$ be the trivial $2$-component spatial graph, where $U_{1}$ is a spatial bouquet with $m$ loops and $U_{2}$ is a spatial bouquet with $n$ loops. Since $B_{i}$ is homeomorphic to $U_{i}$ ($i=1,2$), by Lemma \ref{forklore} it follows that $B_{1}\cup B_{2}$ is a band sum of Hopf links and $U_{1}\cup U_{2}$. Then in the same way as the proof of Lemma \ref{delta_nh1}, all of the Hopf chords whose associated edges belong to the same component can be removed by Delta moves and ambient isotopies. Moreover, by Lemma \ref{delta_lemma} and Lemma \ref{delta_lemma2}, we can deform the band sum of Hopf links so that all of the Hopf chords which are joining the loops $z_{i}$ and $w_{j}$ are parallel for each pair of $i$ and $j$, each of them have no twists of the associated bands as illustrated in Fig. \ref{Hopf_chord5} (1) or each of them has just a half twist of the associated bands as illustrated in Fig. \ref{Hopf_chord5} (2), which depends on the sign of ${\rm lk}(z_{i},w_{j})$, and therefore the number of such Hopf chords equals the absolute value of ${\rm lk}(z_{i},w_{j})$. 

\begin{figure}[htbp]
      \begin{center}
\scalebox{0.35}{\includegraphics*{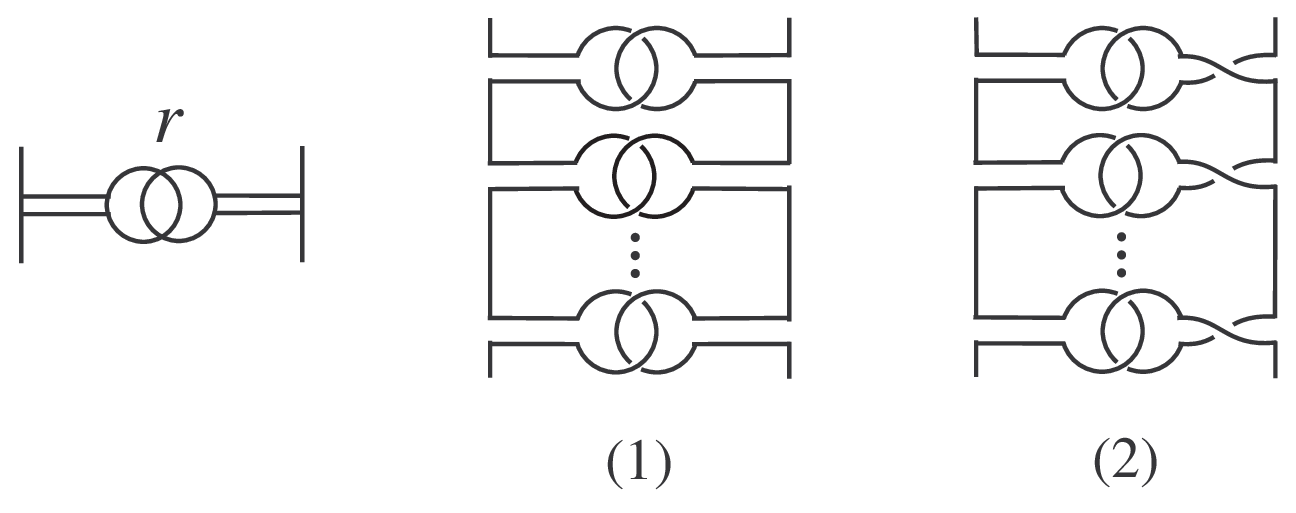}}
      \end{center}
   \caption{}
  \label{Hopf_chord5}
\end{figure} 

Recall that the sequence of elementary divisors of $M_{G}\left({\mathcal Z},{\mathcal W}\right)$ are $d_{1},d_{2},\ldots,d_{l}$. This means that $M_{G}\left({\mathcal Z},{\mathcal W}\right)$ is transformed into the diagonal $(m,n)$-matrix whose $(i,j)$-entry is $d_{i}$ if $i=j$ and $0$ if $i\neq j$ by elementary transformations: (i) exchanging two columns (resp. rows), (ii) multiplying a column (resp. row) by $(-1)$ and  (iii) adding a multiple of a column (resp. row) to another column (resp. row). Exchanging two columns (resp. rows) corresponds to exchanging two bases of $H_{1}(G_{2};{\mathbb Z})$ (resp. $H_{1}(G_{1};{\mathbb Z})$). The multiplication of the $j$th column by $(-1)$ can be realized by a deformation illustrated in Fig. \ref{Hopf_chord7} up to Delta moves and ambient isotopies under changing the base $w_{j}$ into the new base $-w_{j}$, where for an integer $r$, a symbol mark as in the first figure from the left in Fig. \ref{Hopf_chord5} denotes the parallel $r$ Hopf chords which are joining the loops $z_{i}$ and $w_{j}$ as illustrated in Fig. \ref{Hopf_chord5} (1) or (2), which depends on the sign of ${\rm lk}(z_{i},w_{j})$. The multiplication of the $i$th row by $(-1)$ also can be realized by a similar deformation under changing the base $z_{i}$ into the new base $-z_{i}$. Adding a multiple of the $p$th column to the $q$th column ($p\neq q$) can be realized by a deformation illustrated in Fig. \ref{Hopf_chord6} up to Delta moves and ambient isotopies under changing the base $w_{q}$ into the base $w_{q}+w_{p}$. Adding a multiple of the $p'$th row to the $q'$th row ($p'\neq q'$) also can be realized by a similar deformation under changing the base $z_{q'}$ into the base $z_{q'}+z_{p'}$. Therefore by applying Lemma \ref{delta_lemma} and Lemma \ref{delta_lemma2} if necessary, there exist a basis $\widetilde{\mathcal Z}=\left\{\tilde{z}_{1},\tilde{z}_{2},\ldots,\tilde{z}_{m}\right\}$ of $H_{1}(G_{1};{\mathbb Z})$ and a basis $\widetilde{\mathcal W}=\left\{\tilde{w}_{1},\tilde{w}_{2},\ldots,\tilde{w}_{n}\right\}$ of $H_{1}(G_{2};{\mathbb Z})$ such that all loops of $B_{1}$ are regarded as $\tilde{z}_{1},\tilde{z}_{2},\ldots,\tilde{z}_{m}$, all loops of $B_{2}$ are regarded as $\tilde{w}_{1},\tilde{w}_{2},\ldots,\tilde{w}_{n}$, all Hopf chords which are joining the loops $\tilde{z}_{i}$ and $\tilde{w}_{j}$ are parallel for each pair of $i$ and $j$ and ${\rm lk}(\tilde{z}_{i},\tilde{w}_{i})=d_{i}\ (i=1,2,\ldots,l)$. 

\begin{figure}[htbp]
      \begin{center}
\scalebox{0.45}{\includegraphics*{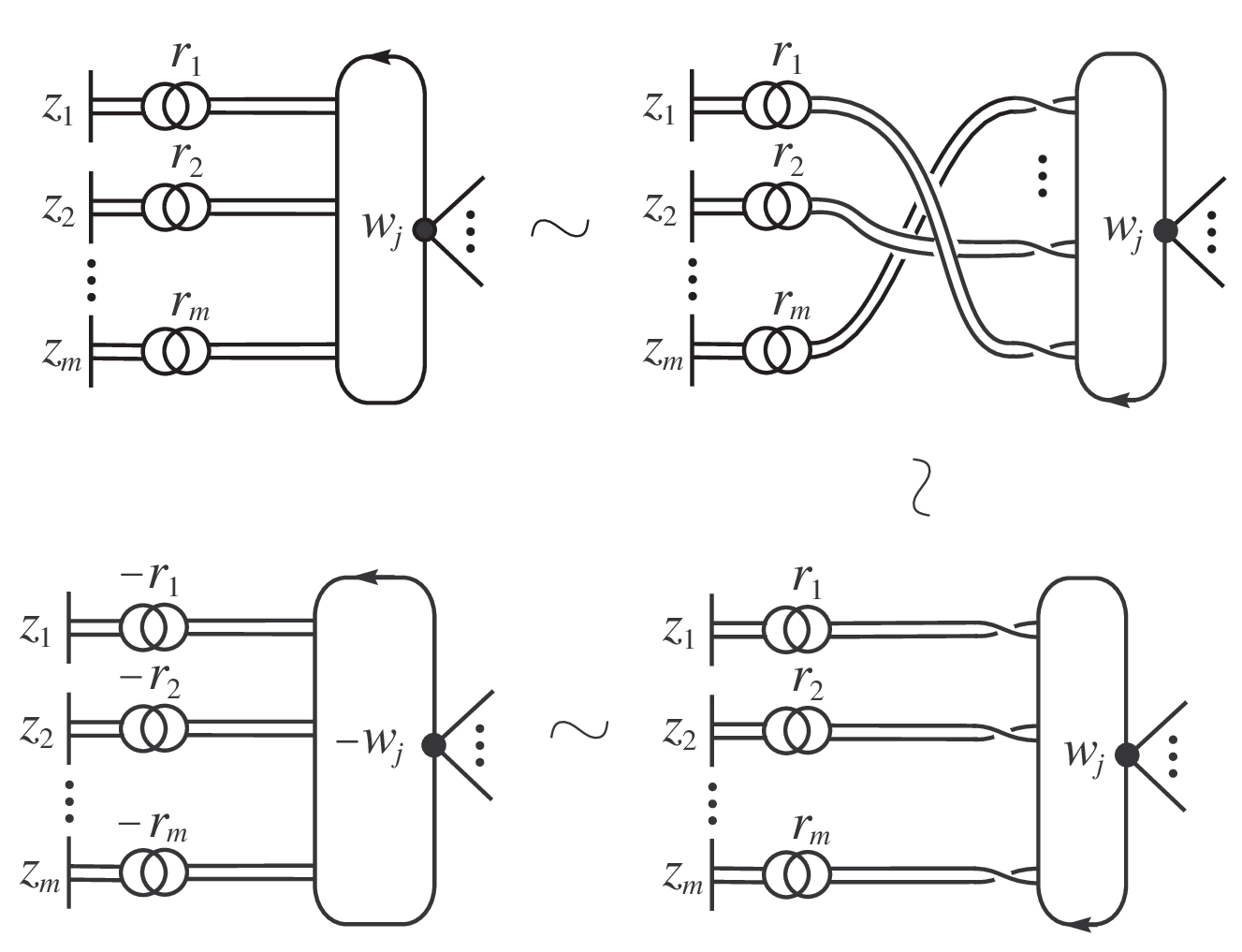}}
      \end{center}
   \caption{}
  \label{Hopf_chord7}
\end{figure} 
\begin{figure}[htbp]
      \begin{center}
\scalebox{0.45}{\includegraphics*{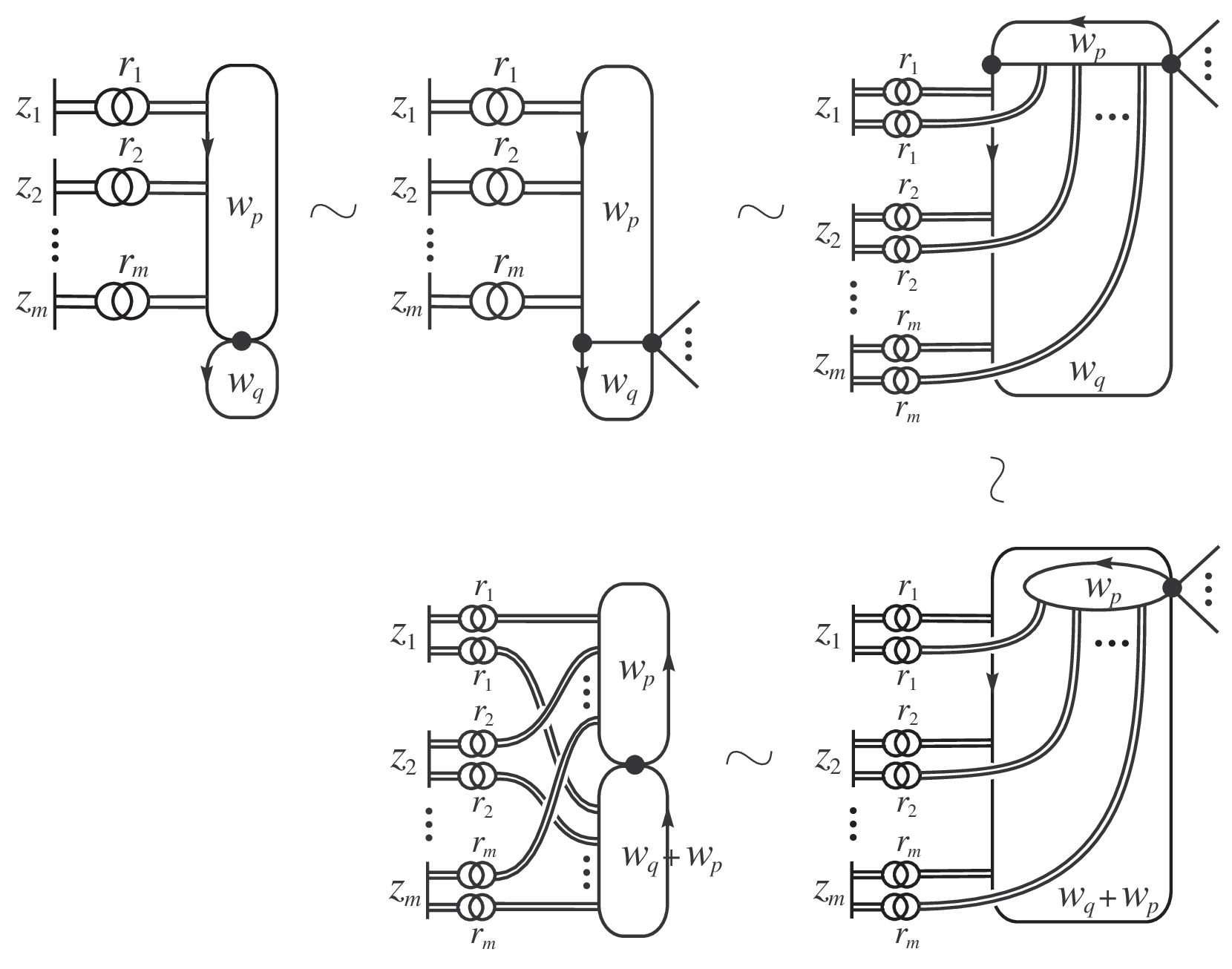}}
      \end{center}
   \caption{}
  \label{Hopf_chord6}
\end{figure} 

Next we deform $G'$ to a similar canonical form of a band sum of Hopf links and the trivial $2$-component spatial graph $U_{1}\cup U_{2}$ by Delta moves, edge contractions, vertex splittings and ambient isotopies, where $U_{1}$ is a spatial bouquet with $m$ loops and $U_{2}$ is a spatial bouquet with $n$ loops. Since ${\rm lk}(G'_{1},G'_{2})=\left\{d_{1},d_{2},\ldots ,d_{l}\right\}$, by applying Lemma \ref{delta_lemma}, the Hopf chords for $G$ and those for $G'$ are transformed into each other by Delta moves and ambient isotopies. Thus $G$ and $G'$ are transformed into each other by Delta moves, edge contractions, vertex splittings and ambient isotopies. This completes the proof. 
\end{proof}

\begin{Remark}
We refer the reader to \cite{MN89} for a complete classification of oriented links and \cite{taniyama95}, \cite{MT97}, \cite{ST03} for a complete classification of spatial embeddings of a graph up to Delta moves and ambient isotopies. 
\end{Remark}


%
{\normalsize
}

\end{document}